\documentclass[11pt,a4paper]{article}
\usepackage[margin=1.6in]{geometry}

\usepackage{graphicx,float,psfrag,epsfig,color}
\usepackage{comment}
\usepackage{subcaption}

  \usepackage{microtype}
  \usepackage[pdftex,pagebackref=true,colorlinks]{hyperref}
  \hypersetup{linkcolor=[rgb]{.7,0,0}}
  \hypersetup{citecolor=[rgb]{0,.7,0}}
  \hypersetup{urlcolor=[rgb]{.7,0,.7}}

\usepackage{amsthm}
\usepackage{amsfonts,dsfont}
\usepackage{amsmath}
\usepackage{amssymb}
\usepackage{graphicx}
\usepackage{epstopdf}
\usepackage{lineno}
\usepackage{setspace} 
\usepackage{verbatim}
\usepackage{hyperref}
\usepackage{cite}
\usepackage{authblk}
\usepackage{color}
\usepackage{mathtools}

\newtheorem{theorem}{Theorem}
\newtheorem*{theorem*}{Theorem}

\newtheorem{lemma}[theorem]{Lemma}

\newtheorem{dfn}{Definition}

\newtheorem*{thm}{Theorem}
\newtheorem*{hyp}{Hypothesis H}

\def\prob{\mathbb{P}}
\def\ve{\varepsilon}

\newcommand{\E}{\mathbb{E}}
\newcommand{\pp}{\mathbb{P}}
\newcommand{\eps}{\varepsilon}

\newenvironment{enumerate*}{
\begin{enumerate}
  \setlength{\itemsep}{5pt}
  \setlength{\parskip}{0pt}
  \setlength{\parsep}{0pt}
}{\end{enumerate}}

\newenvironment{itemize*}{
\begin{itemize}
  \setlength{\itemsep}{5pt}
  \setlength{\parskip}{0pt}
  \setlength{\parsep}{0pt}
}{\end{itemize}}

\begin{document}

\title{
{\Large {\bf Concentration of the number of solutions of random planted CSPs and Goldreich's one-way function candidates}\vspace{.4cm}
}} 
\author{Emmanuel Abbe\thanks{Program in Applied and Computational Mathematics, and EE Department, Princeton University, Princeton, NJ. Email: eabbe@princeton.edu.} \and Katherine Edwards\thanks{Department of Computer Science, Princeton University, Princeton, NJ. Email: ke@princeton.edu.} }

\maketitle
\date{}
\begin{abstract}
This paper shows that the logarithm of the number of solutions of a random planted $k$-SAT formula concentrates around a deterministic $n$-independent threshold.
Specifically, if $F^*_{k}(\alpha,n)$ is a random $k$-SAT formula on $n$ variables, with clause density $\alpha$ and with a uniformly drawn planted solution, there exists a function $\phi_k(\cdot)$ such that, besides for some $\alpha$ in a set of Lesbegue measure zero, we have 
$ \frac{1}{n}\log Z(F^*_{k}(\alpha,n))  \to \phi_k(\alpha)$ in probability, where $Z(F)$ is the number of solutions of the formula $F$.  This settles a problem left open in Abbe-Montanari RANDOM 2013, where the concentration is obtained only for the expected logarithm over the clause distribution. The result is also extended to a more general class of random planted CSPs; in particular, it is shown that the number of pre-images for the Goldreich one-way function model concentrates for some choices of the predicates. 
\end{abstract}

\maketitle
\date{}



\thispagestyle{empty}

\newpage
\pagenumbering{arabic}


\section{Introduction}
This paper investigates concentration phenomena for the number of solutions in random planted random constraint satisfaction problems (CSPs) and the Goldreich one-way function candidate.  

A large body of works have studied phase transition phenomena for satisfiability in random CSPs. 
For uniform\footnote{The model may have a fixed but uniform number of constraints, or a Binomial or equivalent form.} models, the probability of being satisfiable often tends to a step
function as $n$ tends to
infinity, jumping from 1 to 0 when the constraint density crosses a
critical threshold. For random $k$-XORSAT the existence of such a
critical threshold is proved \cite{2-xorsat,dubois,cuckooArxiv,pittel-xorsat}.
For random $2$-SAT, the threshold is proved in \cite{Chvatal,DeLaVega,Goerdt}. 
For random $k$-SAT, $k \geq 3,$ the existence of an $n$-dependent
threshold is proved in \cite{FrBo99}, and the satisfiability threshold conjecture states that this threshold is $n$-independent for all $k$. 
Recently, the conjecture was settled for $k$ large enough \cite{sly_sat}, while 
upper and lower bounds are known to match up to a 
term that is of relative order $k\,2^{-k}$ as $k$ 
increases \cite{AchlioetalNature,coja_asym}. Moreover, 
phase transition phenomena were also studied for a broad family of other CSPs, see for example \cite{AchlioptasTwoCol,MRT09,AchlioetalNature} and references therein.

The counting problem for random formulas has also received attention recently. 
In \cite{AbMo13RSA}, a concentration result is obtained for the number of solutions: at a fixed clause density $\alpha$, the number of solutions of a random $2$-SAT formula concentrates in the logarithmic scale to a deterministic $n$-independent threshold for almost every $\alpha$. This result is extended for $k\geq 3$ for all clause densities having an UNSAT probability decaying fast enough (with a mild logarithmic decay being enough), which is conjectured to take place up to the SAT threshold. 
This result is obtained in two parts. 
First, as was shown earlier in \cite{AchlioptasGeom,AMarxiv}, the property that a random $k$-SAT formula has a number of solution bounded by $2^{n \phi}$, for a fixed $\phi$, has a phase transition with an $n$-dependent threshold, proved \`a la Friedgut. 
This is then turned into a concentration result for the number of solutions in \cite{AbMo13RSA} by showing that the limit for $\frac{1}{n} \E \log(1 + Z(F))$ exists, where $Z(F)$ is the number of solutions of the random formula.
Observe that this gives an $n$-independent threshold for the concentration.  
The key tool in establishing this limit is the interpolation
method, first introduced in \cite{GuerraToninelliLimit}
for the Sherrington-Kirkpatrick model, and subsequently generalized and extended in 
\cite{FranzLeone,FranzLeoneToninelli,PanchenkoTalagrand,BayatiInterpolation,AbMo13RSA}.
Note however that the use of ``$1+$'' in the logarithm above (to obtain a well defined quantity) is responsible for the difficulty in obtaining the concentration for all clause densities when $k \geq 3$.

In this paper, we consider CSPs that have a planted solution and study the counting problem for random ensembles. Planted CSPs are a rich ground for studying combinatorial optimization problems motivated by `real-world' applications, such as in coding theory, community detection, or cryptography, where a solution typically does exist but where the problem is to identify how many other solutions are there, or how hard is it to recover the planted solution.   
Planted ensembles were investigated in \cite{barthel,haanpaa,AAAI,21,jia,barriers}, and at high density in \cite{altarelli,vilenchik,feige2006complete}, and relationships between planted random CSPs and their non-planted counterparts in the satisfiable phase were studied in \cite{barriers,quiet1,quiet2}.

It was shown recently in \cite{AbMo13} that for a broad class of random planted CSPs, the logarithm of the number of solution concentrates to an $n$-independent deterministic threshold for almost every clause density. In particular, this covers $k$-SAT for all $k$'s. Hence, the planting allows to circumvent the issues of establishing the limit of the log-partition function, since the latter is well defined due to the planting (no need for the ``$1+$'' term discussed above). However, the planting also introduces asymmetry in the model, which lead \cite{AbMo13} to a weaker concentration result: the concentration is obtained with respect to the graph ensemble but is taken {\it in expectation} over the clause distribution.

Let us explain this nuance more precisely for random $k$-SAT. A random planted formula is defined in this case by drawing first a random uniform solution $x^0$, and independently, a random 3-hypergraph $G=([n],E)$ at a fixed edge density. The random clauses are then defined for each edge $e \in E(G)$ by drawing a negation pattern $s_e$ uniformly at random within the set of negations patterns that preserve $x^0$ as a planted solution. Specifically, the clause for edge $e$ is defined by $y[e] \neq s_e$ (where $y[e]$ is an assignment of literals to the variables associated with $e$). 
Note that this is indeed equivalent to requiring that the OR of the variables in $y[e]$ negated with the pattern $x_e$ is 1. Consider now $$\phi_n:= \frac{1}{n}\log (Z(F^{(0)})),$$ where $F^{(0)}$ is the random planted formula. 
In \cite{abbetoc}, it is shown that $\E_s \phi_n$, the expectation of $\phi_n$ taken over the variables $s=\{s_e\}_{e \in E(G)}$, concentrates in probability (with respect to the drawing of $G$) to a deterministic $n$-independent value.\footnote{Note that the variables $s=\{s_e\}_{e \in E(G)}$ depend on the planted assignment. For a deterministic kernel $Q$, this means in expectation over the planted assignment.} It was left open to obtain concentration with respect to the drawing of $s$ as well. In particular, the martingale argument used in \cite{abbetoc} fails in this case, since the fluctuations are not bounded, and the application of Friedgut's theorem is mitigated by the lack of symmetry caused by the planting. 

We resolve in this paper the above problem left open in \cite{abbetoc} and show that for almost every $\alpha$, there exists an $n$-independent value $\phi(\alpha)$ such that 
\begin{align*}
\frac{1}{n}\log (Z(F^{(0)})) \to \phi(\alpha) \,\,\, \text{ in probability,}
\end{align*} 
closing the concentration problem. 
The main tool is based on Bourgain's result from the appendix of \cite{FrBo99}. 
The result is then generalized to a broad class of planted CSPs, and a new a application to Goldreich's one way function \cite{goldreich-planted} is investigated. 


The Goldreich one-way function candidate is defined from a $k$-hypergraph $G$ on $n$ vertices and $m$ hyperedges and a {\it fixed} predicate function $\chi: \{0,1\}^k \to \{0,1\}$. 
The function takes an input $x\in \{0,1\}^n$ and, evaluating $\chi$ at each of the $m$ $k$-tuples selected by the hyperedges of $G$, produces an output in $\{0,1\}^m$. 
In \cite{goldreich-planted}, $G$ is proposed to be drawn at random with an edge density $m/n$, and the choice of predicates is further discussed in \cite{BoQi09}.  Note that both $m=\omega(n)$ and $m=\Theta(n)$ are potential candidates \cite{goldreich-planted,BoQi09}. Defining the rate of the one-way function by $m/n$, it is interesting to understand for what rates (in addition to what predicates) is the function possibly one-way, in particular, for the case of $m/n=\alpha$ constant. A natural approach would be to relate this question to the structure of the solution space of the underlying CSP, starting with its size, and hypothetically with the condensation \cite{OurPNAS,cond_color} and freezing of the solution clusters \cite{frozen} phenomena.\footnote{In the non-planted models, these phenomena have been associated with computational barriers for satisfiability.} In particular, the function $\phi(\cdot)$ is expected to have a kink at the condensation threshold for various CSPs \cite{andrea_comm}, and this may indicate a behavioral changes for the hardness of the one-way function. In this paper, we investigate the most basic question towards such considerations: does the function $\phi$ even exist? Namely, does the normalized logarithm of the number of pre-images concentrates for some/all predicates?


We answer this question by the affirmative for a certain class of predicates. Interestingly, it is not obvious that this class of predicates overlaps with the class of predicates that precludes the non-hardness conditions introduced in \cite{BoQi09} for large clause densities. We hence leave an open problem: can one obtain concentration and hardness at the same time, or is hardness related to the non-concentration? We believe that the former is true and that our proof technique stumbles on technicalities, but we cannot resolve this argument.

\section{Models}

\subsection{CSPs arising from satisfiability problems}\label{sec:satmodel}
We first describe a class of constraint satisfaction problems.
Let $V = \{v_1,\dots,v_n\}$ be a set of Boolean variables, and fix an integer $k\geq 2$.
An instance $F$ of a CSP consists of a $k$-uniform multi-hypergraph $(V,E)$ (that is, all edges have cardinality $k$ and we allow parallel edges),
and a family of \emph{clause functions} $\chi_e: \{0,1\}^k \rightarrow \{0,1\}$ for each $e\in E$.
A \emph{$k$-clause} comprises an edge $e$ and its corresponding function $\chi_e$.
We'll sometimes call $F$ a formula. 
The form of the clause function depends on the type of satisfiability problem we are interested in (for the moment, SAT, NAESAT or XORSAT).
Let $y[V]$ denote an assignment $y_1,\dots,y_n$ of Boolean values to the variables in $V$, and $y[e]$ its restriction to the $k$ variables in $e$.
By $\chi_e(y[e])$ we mean the result of evaluating $\chi_e$ on the $k$ values in some \emph{fixed} order (for the moment, the actual choice of ordering of variables in edges isn't important but it will be when we consider certain planted models in Section \ref{sec:planted}.)
This model naturally captures familiar satisfiability problems:

\begin{itemize}
	\item in $k$-SAT, we have $\chi_e(y[e]) = 1 \iff y[e]\neq x_e$ where $x_e \in \{0,1\}^k$ represents a particular forbidden pattern,
	\item in $k$-NAESAT, we have $\chi_e(y[e]) = 1 \iff y[e]\notin \{x_e, \overline{x_e}\}$ where $\overline{x_e}$ is the result of flipping each component of $x_e$,
	\item in $k$-XORSAT, we have $\chi_e(y[e]) = 1 \iff \oplus_i y_i = x_e$ where now $x_e\in \{0,1\}$.
\end{itemize}

An assignment which satisfies all clauses in $F$ is called a satisfying assignment (or solution) for $F$.
Let $C_k(n)$ be the set of all possible $k$-clauses on $V$ and write $N = |C_k(n)|$; in $k$-SAT for example we have $N=\binom{n}{k}2^k$.
We use the \emph{binomial model} for a random CSP with clause density $\alpha\in [0,N/n]$, and draw a random formula $G$ as follows:
\footnote{ 
One could also consider the \emph{uniform model}, wherein $G(n,\alpha)$ is chosen uniformly from those vectors $x \in \{0,1\}^N$ with $|x| = \alpha n$, where $ {\alpha n}/N = p$.
The models are essentially equivalent, and we mostly focus here on the binomial model. }

\begin{equation}
\text{ $\bullet$ include in $G$ each clause in $C_k(n)$ with probability $p={\alpha n}/N$.}
\end{equation}
Let $G(n,\alpha)$ denote a formula obtained by this process.
The formula $G(n,\alpha)$ can be viewed as a random element of $\{0,1\}^N$, drawn according to the product measure $\mu_p$.
\footnote{To see the correspondance, identify each of the $N$ components with a clause and set it to $1$ if and only if the clause is present in the formula.}
That is, for each $x\in \{0,1\}^N$ we have $\mu_p(x) := \prob[G(n,\alpha) = x] = p^{|x|}(1-p)^{N-|x|}$ (where $|x|$ denotes the number of nonzero components). 
Now consider a procedure to sample a \emph{planted} CSP $F$. 
\begin{align}\label{proc:planted}
	&\bullet \text{ Sample $v^0 \in \{0,1\}^n$ uniformly at random.}\\
	\nonumber &\bullet \text{ Then include in $F$ each $k$-clause which is satisfied by $v^0$ independently with probability}\\ 
	\nonumber &\hspace{15pt} \text{$ p={\alpha n}/N$.}
\end{align}
We use the notation $F(n,\alpha)$ to denote a formula obtained by this process.
By construction such a formula is always satisfied by the assignment $v_i=v_i^0$; the vector $v^0$ is known as the planted solution.
Let $Z(F)$ denote the cardinality of the set of assignments to $v_1,\dots,v_n$ which satisfy $F$.
Notice that we always have $Z(F(n,\alpha) \geq 1$ by construction.
Again we view $F(n,\alpha)$ as an element of $\{0,1\}^N$ but 
observe that the distribution in this case satisfies $\mu_p(F) = \prob [F(n,\alpha) = F] = \tfrac{Z(F)}{2^n}p^{|F|}(1-p)^{\binom nk (2^k-1) - |F|} $ so $\mu_p$ is not a product measure here.

\subsection{CSPs arising from Goldreich's one-way function candidate}\label{sec:goldmodel}
The CSPs introduced in the previous section have clause functions taking a few specific forms.
In these examples a satisfying variable assignment $y[V]$ satisfies $\chi_e(y[e])=1, \forall e$,
and the clause functions on individual edges are independent of one another.
Our concentration results can be extended to a related class of CSPs which are related to Goldreich's proposed one-way function \cite{Go00}.
The idea is that we can consider CSPs with arbitrary clause functions if the clauses on different edges are related in a specific way.

In \cite{Go00} Goldreich proposed a candidate one-way function family which exploits the difficulty of recovering a solution to a form of planted CSP.
\footnote{One-way functions are important objects in cryptography and complexity theory. Intuitively these are functions that are computationally easy to evaluate, but hard to invert. For a thorough discussion see \cite{Go05}, \cite{Go00}. }
Goldreich's original proposition was that the following function $f$ is one-way.
As always we work with the variable set $V = \{v_1,\dots,v_n\}$.
\begin{itemize}
	\item Select a predicate $\chi: \{0,1\}^k \rightarrow \{0,1\}$ uniformly at random from the set of all such Boolean functions.
	\item Draw a sparse Erd\H{o}s-R\'{e}nyi $k$-uniform multi-hypergraph $(V,E)$ with $m$ edges $e_1,\dots,e_m$.
	\item $f:\{0,1\}^n \rightarrow \{0,1\}^m$ is the function with $f(x)_i = \chi(x[e_i])$, i.e. the $i$th output bit is the result of evaluating $\phi$ on the $k$ (ordered) values assigned to the edge $e_i$.
\end{itemize}
More precisely, Goldreich conjectured that $f$ is one-way in the setting where $k=O(\log n)$ and $m=n$, and the graph is a sufficiently good expander, 
for most choices of the predicate $\chi$ which is randomly selected and hard-wired into $f$.

With this in mind we can define a class of planted CSPs generated by the following procedure.
\begin{align}\label{proc:gold}
	\nonumber &\bullet \text{ Select a predicate $\chi: \{0,1\}^k \rightarrow \{0,1\}$.}\\
	 &\bullet \text{ Sample $v^0 \in \{0,1\}^n$ uniformly at random.}\\
	\nonumber &\bullet \text{ Then include in $F$ each $k$-clause of the form $e$ with $\chi(y[e]) = \chi(v^0[e])$, }\\
	\nonumber & \text{ \hspace{5pt} with probability $p={\alpha n}/N$.}
\end{align}
Here the edges are ordered subsets of $V$, and so $N = \binom{n}{k}k!$.


\section{Overview of results}

Recall that for is a CSP formula $F$ (planted or not) we denote by $Z(F)$ the number of satisfying assignments for $F$.
If $\phi \in [0,1]$ we write $Q_n(\alpha,\phi) := \prob [Z(F(n,\alpha)) < 2^{n\phi}]$.

\subsection{Concentration of the number of solutions of planted satisfiability CSPs}

Our main result is the following theorem, which states that for fixed $\alpha\geq 0$ the logarithm of the number of solutions of a random planted formula concentrates, closing the problem left open in \cite{AbMo13RSA}. 
Note that this clears the concentration problem in its most general form: the exponent of the number of solutions of a random planted SAT formula can be asymptotically predicted with an $n$-independent value and for any $k \geq 2$ (small or large). The only part that could be further generalized is the fact that the result does not hold for a countable set of ``bad'' $\alpha$'s, but it is unclear whether this is a technicality or not. The formal result reads as follows.  

\begin{theorem}\label{thm:main}
For every $k\geq 2$, there exist a countable set $\mathcal D$ and a function $\phi_s: [0,\alpha^*] \to [0,1]$ such that for every $\alpha \notin \mathcal D$ and every $\epsilon > 0$,
$$\lim_{n\to \infty} Q_n(\alpha, \phi_s(\alpha)-\epsilon) = 0 $$
$$\lim_{n\to \infty} Q_n(\alpha, \phi_s(\alpha)+\epsilon) = 1 $$
\end{theorem}

In \cite{AbMo13}, it was shown that this quantity concentrates when the expectation is taken over the clause distribution. 
We use this result in the proof of Theorem \ref{thm:main}. 

\begin{theorem}\label{thm:conv}\cite{AbMo13}
For every $k\geq 2$, for every $\alpha\in [0,\alpha^*]$ the sequence
$$ \psi_n(\alpha) := \frac 1n \mathbb E [\log Z(F(n,\alpha))]$$
converges almost surely to a limit $\phi_s(\alpha)$.
\end{theorem}

As an intermediate step toward Theorem \ref{thm:main}, we will prove that for fixed $\phi\in [0,1]$ there is a sharp threshold density for the property of having fewer than $2^{n\phi}$ solutions (we define these terms in Section \ref{sec:bourgain}).
First, in Section \ref{sec:planted} we prove the following $n$-dependent sharp threshold.

\begin{lemma}\label{lem:sharp}
For every $k\geq 2$ and for every $\phi \in [0,1)$ there exists a sequence $\{\alpha_n(\phi) \}_{n \in \mathbb Z_{>0}}$ such that for every $\epsilon>0$,
$$\lim_{n\to \infty} Q_n(\alpha_n(\phi)-\epsilon, \phi) = 0 $$
$$\lim_{n\to \infty} Q_n(\alpha_n(\phi)+\epsilon, \phi) = 1. $$
\end{lemma}

In fact, we prove Lemma \ref{lem:sharp} for a larger class of planted CSPs, namely those which arise from Goldreich's one-way function candidate \cite{Go00}.
This allows us to deduce the analogous statement of Theorem \ref{thm:main} for certain instances of these CSPs, as well as an $n$-dependent version of it in general.

In Section \ref{sec:freeze} we combine Lemma \ref{lem:sharp} with Theorem \ref{thm:conv} using a technique from \cite{AbMo13RSA} to show that the sequence $\alpha_n(\phi)$ converges. 

\begin{theorem}\label{thm:freeze}
For every $k\geq 2$, there exist a countable set $\mathcal C$ and a function $\phi_s:[0,\alpha^*) \to [0,1]$ such that for each $\phi\in \phi_s([0,\infty))$ there exists $\alpha_s(\phi)$ such that for each $\epsilon>0$,
$$\lim_{n\to \infty}  Q_n(\alpha_s(\phi) - \epsilon, \phi) = 0$$
and
$$\lim_{n\to \infty}  Q_n(\alpha_s(\phi) + \epsilon, \phi) = 1$$
\end{theorem}
We deduce Theorem \ref{thm:main} from Theorem \ref{thm:freeze} in Section \ref{sec:freeze}.

\subsection{Concentration of the number of solutions of CSPs from Goldreich's one-way function candidates}
We now present concentration results for the number of solutions of the CSPs arising from Goldreich's one-way function candidates described in Section \ref{sec:goldmodel}.

If one considers the logarithm of the number of solutions of the one-way function candidate determined by a random graph $G$, a predicate $\chi$ and a uniform input, and takes the average over the input distribution, it is possible to obtain the following concentration result. 
Note that this gives a stronger concentration notion, i.e., almost sure and for every $\alpha$, and imposes no restriction on the choice of $\chi$. 
However, it provides a $n$-dependent threshold and requires averaging over the input distribution.  

\begin{lemma}\label{thm:goldreichcurve}
Let $F(n,\alpha)$ be a formula drawn as in \eqref{proc:gold} 
Then for every $k\geq 2$, there exist a function $\phi_s^n: [0,\alpha^*] \to [0,1]$, namely $\phi_s^n=\E_{G,v^0} \log Z(F(n,\alpha))$, such that for every $\alpha >0$ and every $\epsilon > 0$, the following holds almost surely 
$$\lim_{n \to \infty} (\E_{v^0} \log Z(F(n,\alpha))-\E_{G,v^0} \log Z(F(n,\alpha))) = 0.$$
\end{lemma}

The proof is found in Appendix Section \ref{sec:goldreichcurve}. 

We can dispose of the dependence of $\phi_s$ on $n$ and on the averaging of the input in the previous theorem for certain choices of $\chi$.
We simply need to remark that Theorem \ref{thm:conv} was in fact shown in \cite{AbMo13RSA} to hold for planted formulas $F(n,\alpha)$ which satisfy a certain convexity hypothesis (let's call it $H$ for now), then the proof of the following theorem follows that of Theorem \ref{thm:main} exactly as in Section \ref{sec:freeze}.
To this end, in the Appendix Section \ref{sec:sharpgold} we prove the analogue of Lemma \ref{lem:sharp} for these CSPs.
This allows us to deduce the analogous statement of Theorem \ref{thm:main} for certain instances of these CSPs.

We say that a predicate $\chi: \{0,1\}^k \rightarrow \{0,1\}$ is \emph{balanced} if it evaluates to $1$ on exactly half of the inputs and we say $\chi$ is antisymmetric if $\chi(x) = 1-\chi(\overline{x})$ for some $x \in \{0,1\}^k$.
\begin{theorem}\label{thm:goldreichconvex}
Let $F(n,\alpha)$ be a formula drawn as in \ref{proc:gold}, with a predicate $\chi$ which is antisymmetric and satisfies Hypothesis $H$.
Then for every $k\geq 2$, there exist a countable set $\mathcal D$ and a function $\phi_s: [0,\alpha^*] \to [0,1]$ such that for every $\alpha \notin \mathcal D$ and every $\epsilon > 0$,
$$\lim_{n\to \infty} Q_n(\alpha, \phi_s(\alpha)-\epsilon) = 0 $$
$$\lim_{n\to \infty} Q_n(\alpha, \phi_s(\alpha)+\epsilon) = 1 $$
\end{theorem}

The hypothesis $H$, stated in terms of $\chi$ is as follows.

\begin{dfn}\label{def:h-gold}
Let $M_1(\{0,1\}^k)$ denote the space of probability measures on $\{0,1\}^k$. Let $\ell \geq 1$. Define $\Gamma: M_1(\{0,1\}^k) \rightarrow \mathbb{R}$ by
$$\nu \mapsto \Gamma_{\ell}(\nu) = \frac{1}{2} \sum_{\substack{u^{(1)}, \dots,u^{(\ell)} \in \{0,1\}^k \\ \chi(u^{(1)}) = \dots = \chi(u^{(\ell)})}}  \prod_{i=1}^k \nu(u_i^{(1)},\dots,u_i^{(\ell)})$$
\end{dfn}

\begin{hyp}
For each $\ell \geq 1$, the operator $\Gamma$ is convex in $\nu$.
\end{hyp}
 
Bogdanov and Qiao showed in \cite{BoQi09} that for many choices of $\chi$, Goldreich's function can be inverted with high probability when $m$ is larger than $n$ by a sufficiently large constant factor.
In particular any $\chi$ which is not balanced or whose output correlates with one or two bits of the input is a bad choice when $m=Dn$ for sufficiently large constant $D$.
Their result suggests that if we want the resulting function to be one-way then we may want $\chi$ to be balanced and not correlated with any bit or pair of bits of the input, but it is unclear whether these would be necessary in the regime $m=n$, the one Goldreich originally suggested.

Strictly speaking, the restriction to antisymmetric $\chi$ in Theorem \ref{thm:goldreichconvex} does not seem necessary.
It is a technical condition which arises in the proof of Lemma \ref{lem:sharp}.
We have verified using a computer search that when $k\leq 5$ no antisymmetric function satisfies the balance properties along with Hypothesis $H$ but it remains unclear to us whether such a function can exist in general.




\section{An overview of the proofs}

The main element in our proofs is Lemma \ref{lem:sharp}, whose proof we give in this section.
From there, obtaining Theorem \ref{thm:main} and its analogues is a straightforward argument given in the Appendix Section \ref{sec:freeze}.

\subsection{Sharp thresholds and Bourgain's theorem}\label{sec:bourgain}

Before proceeding to the proof of Lemma \ref{lem:sharp}, we briefly give a bit of background material on sharp thresholds.
A subset $\mathcal A_n\subseteq \{0,1\}^N$ is called a \emph{property}, and we say it is \emph{nontrivial} if $\mathcal A_n\subset \{0,1\}^N$.
Property $\mathcal A_n$ is \emph{monotone increasing} (or simply monotone) if for every $x\in \mathcal A$ and $x\subseteq y$ we have $y\in \mathcal A$. 
(Containment of formulas is defined in the natural way, namely $x\subseteq y$ iff every nonzero component of $x$ is also nonzero in $y$.)
We may drop the subscript $n$ when it is unambiguous or unnecessary.
A property is \emph{symmetric} if there is a transitive permutation group under which it is invariant.
For example, in (unplanted) SAT, the property of being unsatisfiable is monotone and symmetric.

In this section and the next it is convenient to make a slight abuse of notation, and write $F(n,p)$ in place of $F(n,\alpha)$ to stress that clauses are included in $F(n,\alpha)$ according to binomial $(p=\alpha n/N)$ distribution.
For a monotone property $\mathcal A_n \subset \{0,1\}^N$, write $\mu_p(\mathcal A_n) = \sum_{x\in A_n}\mu_p(x) =  \prob [F(n,p) \in \mathcal A_n]$.
It's not difficult to show that if $A$ is a nontrivial property then $\mu_p(A)$ is a strictly increasing and continuous function of $p$.
For $\gamma \in (0,1)$, let $p_n(\gamma)$ be the value which (uniquely) satisfies $\mu_{p_n(\gamma)}(A) = \gamma$.
We say that $\hat{p_n}$ is a threshold probability if 
$$\lim_{n\rightarrow \infty} \prob [F(n,p) \in \mathcal A] = \{
\begin{array}{lr}
	1 & \textrm{ if } p_n \gg \hat{p_n}\\
	0 & \textrm{ if } p_n \ll \hat{p_n}
\end{array}
$$
where the notation $p_n\gg \hat{p_n}$ indicates that $\tfrac{\hat{p_n}}{p_n}\rightarrow 0$ as $n$ diverges.

We make a distinction between properties exhibiting a very rapid transition versus those with a more gradual one.
Formally, we say that $\mathcal A$ has a sharp threshold if for every $\gamma \in (0,1)$ there exists $p_{\gamma} = p_{\gamma}(n)$ such that $\prob [F(n,p_{\gamma}) \in \mathcal A] = \gamma$, and such that for every $\delta>0$,
$$\lim_{n\rightarrow \infty} \prob [F(n,p) \in \mathcal A] = \{
\begin{array}{lr}
	1 & : \textrm{ if } p(n) \geq (1+\delta)p_{\gamma}(n)\\
	0 & : \textrm{ if } p(n) \leq (1-\delta)p_{\gamma}(n)
\end{array}
$$

Equivalently, for $\tau \in (0,1)$ define $p_0,p_1,p_c$ such that $\mu(p_0) = \tau$, $\mu(p_1) = 1-\tau$ and $\mu(p_c) = \tfrac 12$.
The property $\mathcal A$ has a \emph{sharp} threshold if the ratio $\tfrac{p_1-p_0}{p_c}$ tends to $0$. 
The threshold is \emph{coarse} if this ratio is bounded away from $0$, i.e. if there exists some constant $C$ such that for some $\gamma \in (0,1)$ we have $p_{\gamma}\frac{d\mu_p(A)}{dp}|_{p=p_{\gamma}} < C$.
Friedgut and Kalai (see \cite{FrKa96}) showed that in this case it must be true that $p_{\gamma} = o(1)$.


A crucial contribution to the theory of sharp thresholds is due to Friedgut, in the form of a general existence theorem for sharp thresholds (see \cite{Friedgut05}, \cite{FrBo99}). Roughly, the theorem asserts that if a monotone symmetric property has a coarse threshold, then it can be approximated by the property of containing a small fixed subgraph.
We omit the statement of Friedgut's theorem since it does not apply in our setting; introducing a planted solution does away with the symmetry in the properties we are interested in.
Fortunately in the appendix to \cite{FrBo99}, Bourgain gave an analogue of Friedgut's result to nonsymmetric properties as follows.
This is the theorem we will need to apply.

\begin{theorem}[Bourgain \cite{FrBo99}]\label{thm:bourgain}(See also \cite{KrNa06})
Let $\mathcal A_n \subset \{0,1\}^N$ be a monotone property, and $C>0$ constant. 
Suppose $\mu_p$ is the product measure on $\{0,1\}^N$, i.e. $\mu_p(x) = p^{|x|}(1-p)^{N-|x|}$ for every $x$.
Assume that there exists $\gamma\in (0,1)$ such that 
${\mu_{p_{\gamma}}(\mathcal A_n) = \gamma}$ and 
${p_{\gamma}\frac{d\mu_p(\mathcal A_n)}{dp}|_{p=p_{\gamma}} < C}$ and $p = o(1)$. 
Then there exists $\delta = \delta(C) > 0$ such that either
\begin{enumerate}
	\item $\mu_p(x\in \{0,1\}^n : x \textrm{ contains } x'\in \mathcal A_n \textrm{ of size } |x'| \leq 10C) > \delta$, or
	\item there exists $x'\notin \mathcal A_n$ of size $|x'|\leq 10C$ such that the conditional probability satisfies $$\mu_p(x\in \mathcal A_n | x'\subset x) > \gamma+\delta. $$
\end{enumerate}
\end{theorem}

Friedgut's theorem and Theorem \ref{thm:bourgain} provide a framework for finding sharp thresholds that has been widely exploited. 
These theorems typically allow one to prove the existence of a sharp threshold whose value depends on $n$, whereas in many cases the threshold is believed to converge. 
Friedgut's original application was to show that satisfiability for $k$-SAT has a sharp threshold.
He also used the theorem to prove that in hypergraphs, the property of having a perfect matching, as well as $2$-colourability have sharp thresholds.
With Achlioptas in \cite{AcFr99} they proved that $k$-colourability of graphs (for fixed $k$) has a sharp threshold.
Krivelevich and Nachmias in \cite{KrNa06} showed the same for list-colourability of bipartite graphs.
Their proof uses a neat combinatorial trick (due to Alon) of combining Theorem \ref{thm:bourgain} with a theorem of Erd\H{o}s and Simonovits. 
We use a similar approach in the next section.
A comprehensive survey of applications of Friedgut's theorem can be found in \cite{Friedgut05}.

\subsection{An $n$-dependent sharp threshold for planted CSPs}\label{sec:planted}

Here, we prove Lemma \ref{lem:sharp}.
For a fixed $\phi >0$, we are interested in the property $\mathcal A_{\phi} (={\mathcal{A}_{\phi}}_n) = \{F \in \{0,1\}^N; Z(F)< 2^{\phi n}\}$.
Clearly, $\mathcal A_{\phi}$ is monotone increasing.
We will show that it has a sharp ($n$-dependent) threshold.
As before, let $\mathcal A_{\phi} = \{F \in \{0,1\}^N; Z(F)< 2^{\phi n}\}$,
and now let $F=F(n,p)$ denote a CSP obtained as in (\ref{proc:planted}).
(We explain how the proof can be adjusted to handle $F(n,p)$ as in (\ref{proc:gold}) in the Appendix Section \ref{sec:sharpgold}.)
Lemma \ref{lem:sharp} can be restated as follows.
\begin{lemma}\label{lem:sharpproperty}
For a fixed $k$ and $\phi>0$, the property $\mathcal A_{\phi}$ has a sharp threshold.
\end{lemma}


To prove Lemma \ref{lem:sharpproperty} we will apply Bourgain's Theorem (Theorem \ref{thm:bourgain}). 
In the distribution of $F(n,p)$, we do not have the assumption on $\mu_p$ in the hypothesis of Theorem \ref{thm:bourgain}.
To overcome this difficulty we need to consider fixed plantings, and observe that conditioning on the random choice of $v^0$ doesn't change the probability of the property $\mathcal A_{\phi}$.
By total probability,
$$\prob [F(n,p) \in \mathcal A_{\phi}] = \sum_{v\in \{0,1\}^n}\prob [F(n,p) \in \mathcal A_{\phi} | v^0 = v]\prob [v^0 = v].$$                                                                                                                                                                     
Further, for any $v \in \{0,1\}^n$, the conditional probability satisfies
$$\prob [F(n,p) \in \mathcal A_{\phi} | v^0 = v] = \prob[F(n,p) \in \mathcal A_{\phi} | v^0 = 0^n] $$
since the number of satisfying assignments is unchanged by swapping a variable with its negation.

Therefore, if we let $F^0(n,p)$ denote a formula obtained by independently including each $k$-clause which is satisfied by $v^0 = 0^n$ with probability $p$, we have $\prob [F(n,p) \in \mathcal A_{\phi}] = \prob [F^0(n,p) \in \mathcal A_{\phi}]$.
So to prove Lemma \ref{lem:sharpproperty} it is enough to show that $\mathcal A_{\phi}$ has a sharp threshold when $0^n$ is the planted solution.
Now, the space we are working in is $\{0,1\}^{N'}$, where $N' = \binom nk (2^k-1)$, and indeed $\mu_p(F) = p^{|F|}(1-p)^{N'-|F|}$.
For the remainder of the proof, this will be the assumed setting.

We now proceed to prove the sharp threshold.
The idea is to assume for a contradiction that $\mathcal{A}_{\phi}$ has a coarse threshold, and apply Bourgain's theorem.
We closely follow arguments found in \cite{KrNa06} and \cite{AchlioptasGeom}.
Roughly, Bourgain's theorem implies the existence of some fixed small formula $x'$ whose appearance in a random formula increases the probability of having property $\mathcal A_{\phi}$ by a positive amount.
Note that $\mathcal{A}_{\phi}$, while not symmetric, is invariant under relabelings of the variable set (i.e. automorphisms of $\{v_1,\dots,v_n,\neg v_1,\dots,\neg v_n\}$ which map $\{v_1,\dots,v_n\}$ to itself and $\neg v_i$ to the negation of the image of $v_i$, for each $i$).
This property is sometimes called \emph{permutation symmetry}.
Thus, containing a random (relabeled) copy of $x'$ has the same effect on the probability of having $\mathcal A_{\phi}$.
On the other hand, the assumption that the threshold is coarse implies that adding a large number of random clauses does not drastically change the probability of belonging to $\mathcal A_{\phi}$.
We will see that with the addition of a sufficient number of random clauses we can simulate the addition of $x'$.

\begin{proof}[Proof of Theorem \ref{lem:sharpproperty}]
Suppose for a contradiction that $\mathcal A_{\phi}$ has a coarse threshold.
Then there exist $\gamma$, $p_{\gamma} = o(1)$ and $C$ as in Theorem \ref{thm:bourgain}, and so one of the two cases in its conclusion must hold.

\noindent {\bf Case 1:} $\mu_p(x\in \{0,1\}^n : x \textrm{ contains } x'\in \mathcal A_{\phi} \textrm{ of size } |x'| \leq 10C) > \delta$.\\
If the size of a formula $x'$ is $\leq 10C$, then its clauses involve at most $10Ck$ variables.
Since $x'\in \mathcal A_{\phi}$, and it is satisfied by $v^0$, assigning the planted value to the variables appearing in $x'$ and arbitrary values to the other variables yields a satisfying assignment.
It follows that $Z(x') \geq 2^{n-10Ck} > 2^{\phi n}$ for large enough $n$, so $x' \notin \mathcal{A}_{\phi}$.
This proves that Case 1 cannot occur.
\\

\noindent {\bf Case 2:} there exists $x'\notin \mathcal{A_{\phi}}$ of size $|x'|\leq 10C$ such that the conditional probability satisfies $\mu_{p_{\gamma}}(x\in \mathcal A_{\phi} | x'\subset x) > \gamma+\delta$.\\

Clearly $x'$ is satisfied by $v^0$.
Denote by $t\leq 10Ck$ the number of variables appearing in $x'$.
Without loss of generality, assume these variables are $v_1,\dots,v_t$.
For a $t$-tuple $v = (v_{i_1},\dots, v_{i_t})$ of distinct variables, we write $x'(v)$ to denote the result of relabeling each variable $v_j$ in $x'$ to $v_{i_j}$.
Since $\mathcal A_{\phi}$ has permutation symmetry, it follows that for any $t$-tuple $v$, the conditional probability satisfies $\mu_p(x\in \mathcal A | x(v)\subset x) > \gamma+\delta $.
We write $x^*$ to mean the result of taking $x(v)$ after drawing a uniformly random $t$-tuple $v$.
In other words, if a random formula $F^0(n,p_{\gamma})$ is drawn, the union $F^0(n,p_{\gamma}) \cup x^*$ belongs to $\mathcal{A}_{\phi}$ with probability at least $\gamma + \delta$.

Now, since $p_{\gamma}\frac{d\mu_p(\mathcal A_{\phi})}{dp}|_{p=p_{\gamma}} < C$ it follows that $\lim_{\ve \rightarrow \infty} \frac{\mu_{p_{\gamma}+\ve p_{\gamma}}(\mathcal A_{\phi}) - \mu_{p_{\gamma}}(\mathcal A_{\phi})}{\ve p_{\gamma}} < C$. 
Thus, for some $\ve$ we have $\mu_{p_{\gamma}+\ve p_{\gamma}}(\mathcal A_{\phi}) < \gamma + \tfrac{\delta}{2}$.
Further, (by a standard two-round exposure argument) choosing a formula $F^0(n,p_{\gamma} + \ve p_\gamma)$ is equivalent to choosing formulae $F^0(n,p_{\gamma})$ and $F^0(n,\ve'p_{\gamma})$ for some $\ve'$ and taking their union. 
Note that $\ve,\ve'$ don't depend on $n$, since $C$ does not.

Denote by $x^*$ a random copy of $x'$ drawn as above. Then the above tells us that
$$ \prob [ F^0(n,p_{\gamma}) \cup x^* \in \mathcal A_{\phi} ] > \gamma + \delta $$ while
$$ \prob [ F^0(n,p_{\gamma})  \cup F^0(n,\ve'p_{\gamma}) \in \mathcal A_{\phi} ] < \gamma + \tfrac{\delta}{2}. $$
It follows that for some formula $H_0 \in \{0,1\}^N$ we have 
\begin{equation}\label{eq:H}
\prob [ H_0 \cup x^* \in \mathcal A_{\phi} ] - \prob [ H_0 \cup F^0(n,\ve'p_{\gamma}) \in \mathcal A_{\phi} ] > \tfrac{\delta}{2}
\end{equation}

Clearly, $H_0 \notin \mathcal A_{\phi}$.
Let's say that a $t$-tuple of distinct variables $v = (v_{i_1},\dots,v_{i_t}) \in \{v_1,\dots, v_n\}^t$ is \emph{bad} if $Z(H_0 \cup x(v)) < 2^{\phi n}$.
It follows that at least a $\tfrac{\delta}{2}$ fraction of all $\binom nt t!$ $t$-tuples are bad. Let $T$ be the set of bad tuples.
We need the following theorem of Erd\H os and Simonovits \cite{ErSi83}.

\begin{theorem}[Erd\H os and Simonovits]\label{thm:es}
Let $k,t$ be positive integers and $0\leq \gamma \leq 1$. 
There exists $\gamma' > 0$ such that for sufficiently large $n$, if $T\subset [n]^t$ is such that $|T|>\gamma n^t$ then with probability at least $\gamma'$ a random choice of $t$ disjoint $k$-tuples $X_1,\dots X_t$ from $[n]$ satisfies that every $t$-tuple $(x_1,\dots,x_t)$ with $x_i\in X_i$ is bad. 
We say that $X_1,\dots,X_t$ is \emph{$T$-complete}.
\end{theorem}

We will obtain a contradiction from Theorem \ref{thm:es}.
Basically, we will ensure that with high probability, adding $F^0(n,\ve'p)$ to $H_0$ implies adding clauses $C_1,\dots, C_t$, where each clause $C_i$ forces some variable to be set to its planted value, and the set of $k$-tuples of variables in the clauses is $T$-complete.

Consider drawing $t$ random clauses.
Applying Theorem \ref{thm:es} with $\gamma = \tfrac{\delta}{2}$ we find some $\gamma'$ for which the $t$ $k$-clauses are $T$-complete with probability at least $\gamma'$.
Given that they are $T$-complete, the probability that they are each of the form $\chi_e(v_{i_1}\dots v_{i_k}) \neq 1^k$ (in $k$-SAT or $k$-NAESAT case, or of the form $\chi_e(v_{i_1}\dots v_{i_k}) \neq k \mod 2$ in the $k$-XORSAT case) is $2^{-kt}$.
Observe that each clause forces some variable to take the value $0$, except in $k$-XORSAT when $k$ is even and the clause forces some variable to take the value $1$.

We claim that adding $t$ such clauses to $H_0$ yields a formula with $<k^t 2^{\phi n}$ satisfying assignments. 
Indeed, suppose we have a satisfying assignment. 
Then at least one variable, say $c_i$, from each of the $C_i$ must be set to $0$ ($1$ in the even $k$-XORSAT case). 
This is at least as restrictive as containing $x((c_1,\dots,c_t))$, since $x(0^t)$ is satisfied (and in the even $k$-XORSAT case, therefore $x(1^t)$ is also satisfied ). 
But $(c_1,\dots,c_t)$ is a bad tuple so there are fewer than $2^{\phi n}$ ways to extend these to the remaining variables to get a satisfying assignment for $H_0$.

With high probability, $F(\ve' p_{\gamma})$ has $\Theta(\ve' p_{\gamma} \binom nk (2^k-1)) \rightarrow \infty$ clauses.
So if we draw $F^0(n,\ve' p_{\gamma})$ the probability that the clauses added don't include $t$ clauses which force a $0$ variable as above is at most about ${(1-\gamma'2^{-kt})^{ \ve' p_{\gamma} \binom nk (2^k-1) /t}}$, which we can make as small as we like as $n\rightarrow \infty$. In particular, we can assume it is smaller than $\tfrac {\delta}{2}$.
In the event that $F^0(n,\ve' p_{\gamma})$ does include these $t$ clauses $C_1,\dots,C_t$, consider a satisfying assignment of $H_0 \cup C_1 \dots C_t$.
The probability that it satisfies a randomly chosen $k$-clause is $(1-2^{-k})$.
Therefore, in this case the expected value of $Z(H_0 \cup F^0(n,\ve' p_{\gamma}))$ is at most ${k^t 2^{\phi n} (1-2^{-k})^{|F^0(n,\ve' p_{\gamma})|-t} < 2^{\phi n}}$ with high probability. Applying Markov's inequality, we can ensure that with probability greater than $1-\tfrac{\delta}{2}$, the formula $H_0 \cup F^0(n,\ve' p_{\gamma})\in \mathcal A_{\phi}$, contradicting (\ref{eq:H}).
This proves Case 2 cannot occur and completes the proof of the lemma.

\end{proof}

\section{Open problems}
As mentioned in the introduction, it is not obvious that the conditions on the predicate $\chi$ used to obtain concentration (see Definition \ref{def:h-gold}) are compatible with the conditions ruling out easily invertible functions \cite{BoQi09}. The bottleneck here seems to be Hypothesis H, which at a high-level, translates the sub-additivity of the logarithm of the number of solutions (used to obtain concentration) into a local convexity property of the predicate $\chi$. If the convexity property were in fact necessary, then this would be in conflict with the choice of predicates that seem to avoid the undesirable balanceness properties, making the problem curiously tensed between concentration and hardness. It would hence be interesting to show that this a limitation of our current proof technique, unless concentration has anything to do with hardness. 

Another interesting question would be to obtain convergence rates for the convergence in probability. We obtain an exponential  rate in Theorem \ref{thm:goldreichcurve} using martingale arguments, but this does not apply to our results relying on Bourgain. 

Finally, the results in this paper are about the most basic properties of the solution space, namely, it cardinality. It would be interesting to understand rigorously finer properties of the solution space for planted models to deduce proper choices of the rate and predicates for the Goldreich one-way function.  


\section*{Acknowledgement}
We would like to thank R.\ Impagliazzo for suggesting the Goldreich one-way function model to the first author, as well as A.\ Montanari for stimulating discussions. 

\bibliographystyle{abbrv}
\bibliography{threshold,gen2}


\appendix

\section{Freezing the threshold}\label{sec:freeze}
In this section we prove Theorem \ref{thm:freeze}.
The proof essentially follows arguments in \cite{AbMo13RSA}, but we give it here for completeness.


\begin{proof}[Proof of Theorem \ref{thm:freeze}]
For $\alpha \in [0,\alpha^*]$, let $\phi_s(\alpha)$ denote the limit of the sequence $\psi_n(\alpha) = \frac 1n \mathbb E [\log Z(F(n,\alpha))]$ which converges almost surely by Theorem \ref{thm:conv}.

Let $\phi_0 = \phi_s(\alpha_0)$ for some $\alpha_0$.
In view of Lemma \ref{lem:sharp} it is enough to show that the sequence $\alpha_n(\phi_0)$ obtained there converges (unless $\phi_0$ takes one of countably many values).
Suppose that it does not.
Let 
$$\underline{\alpha_0} =  \liminf_{n\to \infty} \alpha_n(\phi_0)$$ and $$\overline{\alpha_0} = \limsup_{n \to \infty} \alpha_n(\phi_0) .$$
By assumption $\underline{\alpha_0}$ and $\overline{\alpha_0}$ disagree.
Then we can choose increasing sequences $\{m_i\}_{i=1}^{\infty}$ and $\{n_i\}_{i=1}^{\infty}$ such that
$$ \lim_{i\to \infty} \alpha_{m_i}(\phi_0) = \overline{\alpha_0}$$
and
$$ \lim_{i\to \infty} \alpha_{n_i}(\phi_0) = \underline{\alpha_0}.$$

Let $\underline{\alpha_0} \leq \alpha \leq \overline{\alpha_0}$. 
Then for sufficiently large $i$ there exists $\epsilon>0$ such that
$$ Q_{m_i}(\alpha, \phi_0) \leq Q_{m_i}(\alpha_{m_i}(\phi_0)-\epsilon, \phi_0) \to 0 \textrm{ as $i\to \infty$} $$ and
$$ Q_{n_i}(\alpha, \phi_0) \geq Q_{n_i}(\alpha_{n_i}(\phi_0)+\epsilon, \phi_0) \to 1 \textrm{ as $i\to \infty$} .$$

Moreover since $\alpha\geq 0$ we have
$$ Q_{m_i}(\alpha,\phi_0) = \prob \left[Z(F(m_i, \alpha)) < 2^{m_i\phi_0}]  = \prob [\tfrac{1}{m_i}\log Z(F(m_i, \alpha)) < \phi_0 \right]$$
and so we have 
$$\lim_{i \to \infty} \mathbb E \left[ \tfrac{1}{m_i} \log Z(F(m_i, \alpha)) \right] \geq \phi_0$$
i.e.
$$ \phi_s(\alpha) \geq \phi_0, $$
since the above expectation $\psi_n(\alpha)$ converges to $\phi_s(\alpha)$.
A similar argument shows that 
$$\lim_{i \to \infty} \mathbb E \left[ \tfrac{1}{n_i} \log Z(F(n_i, \alpha)) \right] \leq \phi_0$$
and so 
$$ \phi_s(\alpha) \leq \phi_0. $$

It follows that the function $\phi_s$ is constant on $(\underline{\alpha_0} , \overline{\alpha_0})$.
Since $\phi_s$ is non-increasing on $[0,\infty)$ it follows from Froda's theorem that there are countably many values $\phi_0$ for which $\alpha_n(\phi_0)$ does not converge.
This completes the proof.
\end{proof}

We are now all set to prove our main theorem, which we restate now for convenience.

\begin{thm}
For every $k\geq 2$, there exist a countable set $\mathcal D$ and a function $\phi_s: \mathbb R_{\geq 0} \to [0,1]$ such that for every $\alpha \notin \mathcal D$ and every $\epsilon > 0$,
$$\lim_{n\to \infty} Q_n(\alpha, \phi_s(\alpha)-\epsilon) = 0 $$
$$\lim_{n\to \infty} Q_n(\alpha, \phi_s(\alpha)-\epsilon) = 0 $$
\end{thm}

\begin{proof}[Proof of Theorem \ref{thm:main}]
Let $\phi_s$ be the function obtained in Lemma \ref{lem:sharp}, and let $\mathcal D$ denote the (countable) set of its discontinuities.
Assume $\alpha \in [0,\alpha^*] \setminus \mathcal D$, and let $\phi_s(\alpha) = \lim_{n \to \infty} \frac 1n \mathbb E [\log Z(F(n,\alpha))]$ as in Theorem \ref{thm:conv}.

Lemma \ref{lem:sharp} implies that for some countable $\mathcal C$, the limit $A(\phi) = \lim_{n\to \infty} \alpha_n(\phi)$ exists for each $\phi \in \phi_s([0,\alpha^*])\setminus \mathcal C$.
Thus, there exists some $\epsilon' < \epsilon$ such that for $\phi^* := \phi_s(\alpha) - \epsilon'$ we have $\alpha_n(\phi^*)$ converges to a limit $A(\phi^*) > \alpha$.
Therefore, there exists $\delta>0$ such that
$$Q_n(\alpha, \phi_s(\alpha)-\epsilon) \leq Q_n(\alpha, \phi^*) \leq Q_n(\alpha_n(\phi^*) - \delta, \phi^*) \to 0.$$

It follows that $Q_n(\alpha, \phi_s(\alpha)-\epsilon) \to 0$ as $n\to \infty$.
A symmetric argument shows that $Q_n(\alpha, \phi_s(\alpha)+\epsilon) \to 1$ as $n\to \infty$.
This completes the proof.
\end{proof}

\section{Proof of Theorem \ref{thm:goldreichcurve}} \label{sec:goldreichcurve}
Finally we now give the proof of Theorem \ref{thm:goldreichcurve}.

\begin{proof}[Proof of Theorem \ref{thm:goldreichcurve}]
We consider the model of \eqref{proc:gold} for the Goldreich one-way function candidate. Let us denote by $X$ a uniformly drawn input in $\{0,1\}^n$ and by $G$ a random hypergraph of fixed density. The output of the Goldreich one-way function candidate is the vector $Y(X,G)=\{\chi(X[e])\}_{e \in E(G)}$. We denote the number of pre-images of this output by $Z(X,G)$. In what follows, we show that the random variable $L(G):=\E_X \log Z(X,G)$ concentrates around its expectation $\E_{G} L(G)$, which depends on $n$. 

For that purpose, we show that for any that $\eps>0$,
\begin{align}
\pp_G \{ | L(G)-\E_{G} L(G) |  \geq n \eps  \} \leq 2 e^{- n \eps^2/2},
\end{align}
which implies that $|L(G)/n-\E_{G} L(G)/n|$ converges almost surely to $0$ from the Borel-Cantelli Lemma.
The above inequality results from a standard application of the Azuma-Hoeffding inequality \cite{azuma}, as used in \cite{abbetoc} for more general models. Since the graph is Erdos-Renyi, we consider equivalently the edges to be drawn uniformly at random (conditioning on the number of edges in the graph). We need to show that, if $e$ is an edge picked uniformly at random and $G \cup e$ is the augmented graph, the increment $L(G)- L(G \cup e)$ is bounded. In fact, 
\begin{align}
L(G)- L(G \cup e) &= \E_X \log \frac{Z(X,G)}{Z(X,G \cup e)} \\
& \leq  \log \E_X \frac{Z(X,G)}{Z(X,G \cup e)} \\
&= - \log \E_X  \E_{U|X} \mathds{1}(\chi(X[e]) =\chi(U[e])) \label{collision}
\end{align}
where $U$ is a random vector uniformly drawn among all vectors $u \in \{0,1\}^n$ such that the output of $u$ is the same as the output of $X$ on the one-way function defined by $G$ and $\chi$. Note that $X$ and $U$ are not independent but exchangeable, i.e., they are independent conditionally on their common output $Y$. Therefore 
\begin{align}
\E_X  \E_{U|X} \mathds{1}(\chi(X[e]) =\chi(U[e])) &= \E_{X,U} \mathds{1}(\chi(X[e]) =\chi(U[e])) \\
&= \E_{X[e],U[e]} \mathds{1}(\chi(X[e]) =\chi(U[e])) \\
&= \sum_{y} \E_{X[e],U[e]|Y=y} \mathds{1}(\chi(X[e]) =\chi(U[e])) \pp\{Y=y \} \\
&= \sum_{y}  (\pp \{S_0 | Y=y \}^2 +  \pp \{ S_1 | Y=y \}^2 ) \pp\{Y=y \}
\end{align}
where $\pp \{ S_i | Y=y \}$ is the probability that $X[e]$ belongs to $\chi^{-1}(i)$ given that $Y=y$, for $i=0,1$.
Since $\pp \{S_0 | Y=y \} +  \pp \{ S_1 | Y=y \} =1$, we have $\pp \{S_0 | Y=y \}^2 +  \pp \{ S_1 | Y=y \}^2 \geq 1/2$, hence  
\begin{align}
\sum_{y}  (\pp \{S_0 | Y=y \}^2 +  \pp \{ S_1 | Y=y \}^2 ) \pp\{Y=y \} \geq 1/2,
\end{align}
and \eqref{collision} is upper-bounded by $\log(2)=1$. 
\end{proof}

\section{A sharp $n$ and $\chi$-dependent threshold for CSPs from Goldreich's functions}\label{sec:sharpgold}

To prove Lemma \ref{lem:sharpproperty} for a Goldreich random CSP as in (\ref{proc:gold}) we need only make a slight modification to the proof in Section \ref{sec:planted}.
First, to put ourselves in the setting where we have a product measure we fix the predicate $\chi$ (we do not need to fix the planted solution as we did above). 
This implies that the threshold we obtain may depend on both $n$ and $\chi$.
As before, let $\mathcal A_{\phi} = \{F \in \{0,1\}^N; Z(F)< 2^{\phi n}\}$,
and now let $F=F(n,p)$ denote a CSP obtained as in (\ref{proc:gold}), with $\chi$ fixed and denote by $v^0$ the planted solution.
The space we are working in is $\{0,1\}^{N}$, where $N =2 \binom nk$, and indeed $\mu_p(F) = p^{|F|}(1-p)^{N-|F|}$.

Lemma \ref{lem:sharp} can be restated as follows.
\begin{lemma}\label{lem:sharppropertygold}
For a fixed $k$ and $\phi>0$, the property $\mathcal A_{\phi}$ has a sharp threshold.
\end{lemma}

The only place in which the proof of Lemma \ref{lem:sharppropertygold} differs from the proof of Lemma \ref{lem:sharpproperty} is in the application of Theorem \ref{thm:es}, but we give the details for completeness.

\begin{proof}[Proof of Theorem \ref{lem:sharppropertygold}]
Suppose for a contradiction that $\mathcal A_{\phi}$ has a coarse threshold.
Then there exist $\gamma$, $p_{\gamma} = o(1)$ and $C$ as in Theorem \ref{thm:bourgain}, and so one of the two cases in its conclusion must hold.

\noindent {\bf Case 1:} $\mu_p(x\in \{0,1\}^n : x \textrm{ contains } x'\in \mathcal A_{\phi} \textrm{ of size } |x'| \leq 10C) > \delta$.\\
If the size of a formula $x'$ is $\leq 10C$, then its clauses involve at most $10Ck$ variables.
Since $x'\in \mathcal A_{\phi}$, and it is satisfied by $v^0$, assigning the planted value to the variables appearing in $x'$ and arbitrary values to the other variables yields a satisfying assignment.
It follows that $Z(x') \geq 2^{n-10Ck} > 2^{\phi n}$ for large enough $n$, so $x' \notin \mathcal{A}_{\phi}$.
This proves that Case 1 cannot occur.
\\

\noindent {\bf Case 2:} there exists $x'\notin \mathcal{A_{\phi}}$ of size $|x'|\leq 10C$ such that the conditional probability satisfies $\mu_{p_{\gamma}}(x\in \mathcal A_{\phi} | x'\subset x) > \gamma+\delta$.\\

Clearly $x'$ is satisfied by $v^0$.
Denote by $t\leq 10Ck$ the number of variables appearing in $x'$.
Without loss of generality, assume these variables are $v_1,\dots,v_t$.
For a $t$-tuple $v = (v_{i_1},\dots, v_{i_t})$ of distinct variables, we write $x'(v)$ to denote the result of relabeling each variable $v_j$ in $x'$ to $v_{i_j}$.
Since $\mathcal A_{\phi}$ has permutation symmetry, it follows that for any $t$-tuple $v$, the conditional probability satisfies $\mu_p(x\in \mathcal A | x(v)\subset x) > \gamma+\delta $.
We write $x^*$ to mean the result of taking $x(v)$ after drawing a uniformly random $t$-tuple $v$.
In other words, if a random formula $F^0(n,p_{\gamma})$ is drawn, the union $F^0(n,p_{\gamma}) \cup x^*$ belongs to $\mathcal{A}_{\phi}$ with probability at least $\gamma + \delta$.

Now, since $p_{\gamma}\frac{d\mu_p(\mathcal A_{\phi})}{dp}|_{p=p_{\gamma}} < C$ it follows that $\lim_{\ve \rightarrow \infty} \frac{\mu_{p_{\gamma}+\ve p_{\gamma}}(\mathcal A_{\phi}) - \mu_{p_{\gamma}}(\mathcal A_{\phi})}{\ve p_{\gamma}} < C$. 
Thus, for some $\ve$ we have $\mu_{p_{\gamma}+\ve p_{\gamma}}(\mathcal A_{\phi}) < \gamma + \tfrac{\delta}{2}$.
Further, (by a standard two-round exposure argument) choosing a formula $F^0(n,p_{\gamma} + \ve p_\gamma)$ is equivalent to choosing formulae $F^0(n,p_{\gamma})$ and $F^0(n,\ve'p_{\gamma})$ for some $\ve'$ and taking their union. 
Note that $\ve,\ve'$ don't depend on $n$, since $C$ does not.

Denote by $x^*$ a random copy of $x'$ drawn as above. Then the above tells us that
$$ \prob [ F^0(n,p_{\gamma}) \cup x^* \in \mathcal A_{\phi} ] > \gamma + \delta $$ while
$$ \prob [ F^0(n,p_{\gamma})  \cup F^0(n,\ve'p_{\gamma}) \in \mathcal A_{\phi} ] < \gamma + \tfrac{\delta}{2}. $$
It follows that for some formula $H_0 \in \{0,1\}^N$ we have 
\begin{equation}\label{eq:H2}
\prob [ H_0 \cup x^* \in \mathcal A_{\phi} ] - \prob [ H_0 \cup F^0(n,\ve'p_{\gamma}) \in \mathcal A_{\phi} ] > \tfrac{\delta}{2}
\end{equation}

Clearly, $H_0 \notin \mathcal A_{\phi}$.
Let's say that a $t$-tuple of distinct variables $v = (v_{i_1},\dots,v_{i_t}) \in \{v_1,\dots, v_n\}^t$ is \emph{bad} if $Z(H_0 \cup x(v)) < 2^{\phi n}$.
It follows that at least a $\tfrac{\delta}{2}$ fraction of all $\binom nt t!$ $t$-tuples are bad. Let $T$ be the set of bad tuples.
We need Erd\H os and Simonovits' Theorem \ref{thm:es} again.


We will ensure that with high probability, adding $F^0(n,\ve'p)$ to $H_0$ implies adding clauses $C_1,\dots, C_t$, where each clause $C_i$ forces some variable to be set to its planted value, and the set of $k$-tuples of variables in the clauses is $T$-complete.

Consider drawing $t$ random clauses.
Applying Theorem \ref{thm:es} with $\gamma = \tfrac{\delta}{2}$ we find some $\gamma'$ for which the $t$ $k$-clauses are $T$-complete with probability at least $\gamma'$.
Given that they are $T$-complete, the probability that they are each of the form 
$\chi(v_{i_1}\dots v_{i_k}) = \chi(v_{i_1}^0\dots v_{i_k}^0)$ 
By the antisymmetry of $chi$, each such clause forces some variable to take the planted value.

We claim that adding $t$ such clauses to $H_0$ yields a formula with $<k^t 2^{\phi n}$ satisfying assignments. 
Indeed, suppose we have a satisfying assignment. 
Then at least one variable, say $c_i$, from each of the $C_i$ must be set to the planted value $c_i^0$.
But $(c_1,\dots,c_t)$ is a bad tuple so there are fewer than $2^{\phi n}$ ways to extend these to the remaining variables to get a satisfying assignment for $H_0$.

With high probability, $F(\ve' p_{\gamma})$ has $\Theta(\ve' p_{\gamma} \binom nk (2^k-1)) \rightarrow \infty$ clauses.
So if we draw $F^0(n,\ve' p_{\gamma})$ the probability that the clauses added don't include $t$ clauses which force a $0$ variable as above is at most about ${(1-\gamma'2^{-kt})^{ \ve' p_{\gamma} \binom nk (2^k-1) /t}}$, which we can make as small as we like as $n\rightarrow \infty$. In particular, we can assume it is smaller than $\tfrac {\delta}{2}$.
In the event that $F^0(n,\ve' p_{\gamma})$ does include these $t$ clauses $C_1,\dots,C_t$, consider a satisfying assignment of $H_0 \cup C_1 \dots C_t$.
The probability that it satisfies a randomly chosen $k$-clause is $(1-2^{-k})$.
Therefore, in this case the expected value of $Z(H_0 \cup F^0(n,\ve' p_{\gamma}))$ is at most ${k^t 2^{\phi n} (1-2^{-k})^{|F^0(n,\ve' p_{\gamma})|-t} < 2^{\phi n}}$ with high probability. Applying Markov's inequality, we can ensure that with probability greater than $1-\tfrac{\delta}{2}$, the formula $H_0 \cup F^0(n,\ve' p_{\gamma})\in \mathcal A_{\phi}$, contradicting (\ref{eq:H2}).
This proves Case 2 cannot occur and completes the proof of the lemma.

\end{proof}

\end{document}